\documentclass[12pt,a4paper]{article}
\usepackage[utf8]{inputenc}
\usepackage[english]{babel}
\usepackage[utf8]{inputenc}
\usepackage[T1]{fontenc}
\usepackage[english]{babel}
\usepackage{fullpage}
\usepackage{graphics}
\usepackage{epsfig}
\usepackage{float}
\usepackage{tabu}
\usepackage{verbatim}
\usepackage{csquotes}
\usepackage{amsmath}
\usepackage{ae}
\usepackage{units}
\usepackage{color}
\usepackage{graphicx}
\usepackage{epstopdf}
\usepackage{subfigure}
\usepackage{bbm}
\usepackage{multirow}
\usepackage{array}
\usepackage[hyphens]{url}
\usepackage{lettrine}
\usepackage{eso-pic}
\usepackage{enumitem}
\usepackage{titlesec}

\usepackage{textcomp}
\usepackage{wasysym}

\usepackage{amssymb}
\usepackage{amsthm}
\usepackage{thmtools}
\usepackage{tikz}
\usetikzlibrary{matrix}
\usepackage{amsmath}
\usepackage{commath}
\usepackage{mathtools}
\usepackage{amsfonts}

\declaretheoremstyle[spaceabove=15pt,spacebelow=15pt,bodyfont=\normalfont\itshape,headfont=\normalfont\bfseries]{thmStyle}
\declaretheoremstyle[spaceabove=15pt,spacebelow=15pt,bodyfont=\normalfont\itshape,notefont=\normalfont\bfseries,notebraces={}{},headformat=\NOTE,headfont=\normalfont\bfseries]{noteThmStyle}
\declaretheoremstyle[spaceabove=15pt,spacebelow=15pt,headfont=\normalfont\bfseries]{proofStyle}
\declaretheoremstyle[spaceabove=15pt,spacebelow=15pt]{defStyle}
\declaretheoremstyle[spaceabove=15pt,spacebelow=15pt,headfont=\normalfont\itshape]{remStyle}
\declaretheoremstyle[spaceabove=15pt,spacebelow=15pt,headfont=\normalfont\itshape\bfseries]{exampleStyle}

\declaretheorem[style=thmStyle,numberwithin=section]{theorem}
\declaretheorem[sibling=theorem,style=noteThmStyle,name=Theorem]{noteTheorem}

\declaretheorem[sibling=theorem,style=thmStyle]{proposition}
\declaretheorem[sibling=theorem,style=thmStyle]{corollary}
\declaretheorem[sibling=theorem,style=thmStyle]{lemma}
\declaretheorem[sibling=theorem,style=defStyle,qed=$\diamond$]{definition}
\declaretheorem[style=remStyle,numbered=no]{remark}

\declaretheorem[style=exampleStyle,qed=$\checked$]{example}

\renewcommand{\i}{\mathrm{i}}
\renewcommand{\d}{\mathrm{d}}
\renewcommand{\varphi}{\phi}
\renewcommand{\epsilon}{\varepsilon}
\renewcommand{\leq}{\leqslant}
\renewcommand{\geq}{\geqslant}

\renewcommand{\O}{\mathcal{O}}
\renewcommand{\ker}{\mathrm{ker}}
\newcommand{\im}{\mathrm{im}}

\newcommand{\Q}{\mathbf{Q}} 		% rational numbers
\newcommand{\R}{\mathbf{R}} 		% real numbers
\newcommand{\C}{\mathbf{C}} 		% complex numbers
\newcommand{\CP}{\mathbf{CP}} 	% complex projective
\newcommand{\SM}{\mathcal{SM}} 	% semi-meromorphic
\newcommand{\E}{\mathcal{E}} 	% smooth
\newcommand{\dbar}{\bar{\partial}} 	% dbar
\newcommand{\e}{\mathrm{e}} 		% exponential

\newcommand{\Res}{\mathrm{Res}}

\newcommand{\supp}{\mathrm{supp}}

\newcommand{\QM}{\mathcal{QM}} 	% quasi-meromorphic

\usepackage{mathrsfs}
\DeclareFontEncoding{LS1}{}{}
\DeclareFontSubstitution{LS1}{stix}{m}{n}
\DeclareMathAlphabet{\mathsc}{LS1}{stixscr}{m}{n}

\renewenvironment{abstract}%
{\begin{center} \bfseries \abstractname \end{center}}%
{\vspace{2 \baselineskip}}%

\titleformat{\section}[hang]
{\normalsize\scshape} % format
{\thesection. } % label
{0.1ex} % sep
{\vspace{0.5ex}\centering}
[\vspace{-0.5ex}]

\titleformat{\subsection}[runin]
{\normalfont\bfseries}
{\thesubsection.}{0.5em}{}
[.]

\usepackage[square,numbers]{natbib}
\title{Residues and currents from singular forms on complex manifolds}
\date{\vspace{-5ex}}
\author{\small{\textsc{Mattias Lennartsson}}}

\begin{document}
\footnotesize

\maketitle
\vspace{0.5ex}
\begin{abstract}
\vspace{-2ex}
Using methods from the theory of residue currents we provide asymptotic expansions of certain divergent integrals on complex manifolds. We express the coefficients in these expansions with the conjugate Dolbeault residue, introduced by Felder and Kazhdan in \cite{Fe}, and define a new residue which we call the Aeppli residue.
\end{abstract}
\vspace{-4ex}
%\tableofcontents
%\newpage
%\pagestyle{empty}

%%%%%%%%%%%%%%%%%%%%%%
\section{Introduction}
%%%%%%%%%%%%%%%%%%%%%%

Suppose $X$ is a compact complex manifold of dimension $d$ and $D\subset X$ is a smooth hypersurface. Motivated by perturbative string theory, in \cite{Fe} Felder and Kazhdan discuss regularisations of divergent integrals of the form
\begin{equation*}
\int_{X}\alpha\wedge\overline{\beta}
\end{equation*}
where $\alpha$ and $\beta$ are $(d,0)$-forms which are smooth on $X\setminus D$, $\alpha$ has a pole along $D$ and $\beta$ has a pole of order one along $D$. In their paper they use cut-off functions, i.e.\! functions $\chi$ which are zero on $D$ and otherwise positive, and prove the asymptotic expansion
\[\int_{\chi\geq\epsilon}\alpha\wedge\overline{\beta}=\log\epsilon\,I_{0}+I_{1}(\chi)+\O(\epsilon)\]
where $I_{0}=\int_{D}\Res\,\alpha\wedge\overline{\Res\,\beta}$ does not depend on the cut-off function (here $\Res$ denotes the classical Leray residue which we discuss later). They also show that $I_{1}(\chi)$ depends linearly on $\chi$ and give an explicit expression for it in terms of the \emph{conjugate Dolbeault residue}, $\Res_{\partial}$, defined in the same paper. In a second paper, \cite{Fe2}, the same authors generalise the results to smooth manifolds and forms which have singularities on submanifolds determined by Morse--Bott functions. In particular they consider the case of a complex hypersurface with normal crossings. They also study analytic continuations of these divergent integrals.

In this paper we take the analytic continuation of divergent integrals as starting point. This means that we have a different method of regularising the divergent integrals and this will give us more explicit formulas. We allow $D$ to be a hypersurface with normal crossings and $\alpha$ and $\beta$ to be semi-meromorphic forms with poles along $D$ of any order. If $s:X\rightarrow L$ is a holomorphic section of some line bundle such that $D=\{s=0\}$ and $|\cdot|$ is a metric on $L$ we define a function by
\begin{equation*}
\lambda\mapsto\int_{X}|s|^{2\lambda}\alpha\wedge\bar{\beta}.
\end{equation*}
This function is a priori only defined for complex numbers $\lambda$ with $\mathrm{Re}\,\lambda$ large enough but we will see that it has a meromorphic extension to $\C$ which is holomorphic when $\mathrm{Re}\,\lambda$ is large enough. We get a Laurent expansion at $0$, cf.\! Theorem \ref{thmPrinc},
\begin{equation}\label{introEq}
\int_{X}|s|^{2\lambda}\alpha\wedge\bar{\beta}=\lambda^{-\kappa}C_{-\kappa}+\dots+\lambda^{-1}C_{-1}+C_{0}+\O(\lambda)
\end{equation}
where $\kappa$ is defined in Section 2. Changing $\alpha\wedge\bar{\beta}$ to $\alpha\wedge\bar{\beta}\wedge\xi$, where $\xi$ is a test function, we get currents $C_{-j}(\xi)$ of bidegree $(d,d)$. We will focus on the leading coefficient $C_{-\kappa}$, which we call the \emph{canonical current} associated to $\alpha\wedge\bar{\beta}$, and we denote it by $\{\alpha\wedge\bar{\beta}\}$. The motivation for this construction comes from the study of residue currents in complex geometry. Then one looks at so called semi-meromorphic forms $\alpha$, i.e.\! locally $\alpha=\widetilde{\alpha}/f$ for some smooth form $\widetilde{\alpha}$ and some holomorphic function $f$ such that $f\not\equiv0$. Given such a form one can use this method to define the \emph{principal value current} $[\alpha]$. We will recall more precisely how this is done in Section 2.

In the third section we discuss cohomological residues. Given a semi-meromorphic $(d,d-1)$-form $\alpha$ on $X$ which is polar along a smooth hypersurface $D$ the conjugate Dolbeault residue $\Res_{\partial}(\alpha)$ is a class in the conjugate Dolbeault cohomology group $H^{d-1,d-1}_{\partial}(D)$, see Definition \ref{definitionConj} below. We then define a new residue, which we call the \emph{Aeppli residue}, and denote it by $\Res_{A}$. Given semi-meromorphic $(d,0)$-forms $\alpha$ and $\beta$ which are polar along $D$ the Aeppli residue $\Res_{A}(\alpha\wedge\bar{\beta})$ is a class in the Aeppli cohomology group $H^{d-1,d-1}_{A}(D)$. We relate these residues to the currents defined from analytic continuations of divergent integrals. The following result relates principal value currents and the conjugate Dolbeault residue.

\begin{noteTheorem}[Theorem A]
For a semi-meromorphic form $\alpha$ which is polar along a smooth hypersurface $D$ we have, for every test form $\xi$,
\begin{align*}
\big\langle\dbar[\alpha],\xi\big\rangle&=\big\langle[\dbar\alpha],\xi\big\rangle+2\pi\i\int_{D}\Res_{\partial}(\alpha\wedge\xi).
\end{align*}
\end{noteTheorem}

In the same spirit we can relate the canonical current to the Aeppli residue. We prove a more general result in Theorem \ref{thmAeppli2} but a special case is the following.

\begin{noteTheorem}[Theorem B]
For semi-meromorphic forms $\alpha$ and $\beta$, polar along a smooth hypersurface $D$, we have for every test form $\xi$,
\[\big\langle\{\alpha\wedge\bar{\beta}\},\xi\big\rangle=-2\pi\i\int_{D}\Res_{A}(\alpha\wedge\bar{\beta}\wedge\xi).\]
\end{noteTheorem}

Theorem A and B concerns the leading coefficient in expansions such as (\ref{introEq}). In Section 4 we use the previous results to describe the other coefficients, see Theorem \ref{thmAsymp} below. One of the main points of Theorem \ref{thmAsymp} is the following informally stated result.

\begin{noteTheorem}[Theorem C]
The coefficient $C_{-r}$ in the asymptotic expansion (\ref{introEq}) depends polynomially of degree $\kappa-r$ on the chosen metric.
\end{noteTheorem}

We finally note that asymptotic expansions similar to (\ref{introEq}) have been studied before, see e.g.\! \cite{Bar,Bar2}, but to our understanding these results are not directly related to our residues.

%%%%%%%%%%%%%%%%%%%%%%
\section{Currents from singular forms}
%%%%%%%%%%%%%%%%%%%%%%

We recall some facts about semi-meromorphic forms and how to define principal value currents from them. In Section 2.2 we define currents from more general forms. Throughout $X$ will be a complex manifold of dimension $d$.

%%%%%%%%%%%%%%%%%%%%%%
\subsection{Semi-meromorphic forms}
%%%%%%%%%%%%%%%%%%%%%%

We denote by $\SM(X)$ the semi-meromorphic forms, i.e.\! forms $\alpha$ which can be written locally as $\alpha=\widetilde{\alpha}/f$ where $\widetilde{\alpha}$ is a smooth form and $f$ a holomorphic function such that $f\not\equiv0$. We write $P(\alpha)$ for the polar set of $\alpha$, which consists of the points where $\alpha$ is not smooth. Given the local description above we get $P(\alpha)\subset\{f=0\}$. For a hypersurface $D$ we write $\mathcal{E}(*D)$ for the semi-meromorphic forms which have a polar set contained in $D$ and $\mathcal{E}^{p,q}(*D)$ for the ones of bidegree $(p,q)$. Since the pole of a semi-meromorphic form is determined locally by a holomorphic function, locally the order of the pole is well defined.

One way to define principal value currents from semi-meromorphic forms is the following cf.\! \cite{And,Ber,Sa}: suppose $\alpha\in\mathcal{E}(*D)$ has a hypersurface $D$ with normal crossings as polar set and $D=\{s=0\}$ where $s:X\rightarrow L$ is a holomorphic section of some line bundle $L$. Let $|\cdot|$ be a metric on $L$ and $\xi$ a test form of complementary degree. The function
\[\lambda\mapsto\int_{X}|s|^{2\lambda}\alpha\wedge\xi\]
is a priori only defined when $\mathrm{Re}\,\lambda\gg1$. One can show, however, that the function has an analytic continuation to $\mathrm{Re}\,\lambda>-\epsilon$ for some $\epsilon>0$. Thus we may define the principal value current $[\alpha]$ by
\[\big\langle[\alpha],\xi\big\rangle=\Big(\int_{X}|s|^{2\lambda}\alpha\wedge\xi\Big)\Big|_{\lambda=0}.\]
The current does not depend on the choice of metric $|\cdot|$ or section $s$.

%%%%%%%%%%%%%%%%%%%%%%
\subsection{Quasi-meromorphic forms}
%%%%%%%%%%%%%%%%%%%%%%

We let $\QM(X)$ denote forms $\omega$ which can be written locally as $\omega=\widetilde{\omega}/f\bar{g}$ where $\widetilde{\omega}$ is a smooth form and $f$ and $g$ are holomorphic functions which are not identically zero. We call these forms \emph{quasi-meromorphic} and they are smooth forms except that they can have real analytic singularities along (local) complex hypersurfaces.

For $\omega\in\QM(X)$ we define its polar set, denoted by $P(\omega)$, as the set of points where $\omega$ is not smooth. When $\omega$ has a polar set contained in a hypersurface $D$ we write $\omega\in\mathcal{E}(*\bar{*}D)$, we call $D$ the polar set even though $\omega$ may be smooth on parts of $D$. We will focus on forms in $\E(*\bar{*}D)$, for some $D$, since it is notationally more convenient. We write $\mathcal{E}^{p,q}(*\bar{*}D)$ for the forms in $\mathcal{E}(*\bar{*}D)$ which have bidegree $(p,q)$.

The polar set of a quasi-meromorphic form has different parts between which we need to distinguish. We define the subset $P^{1,0}(\omega)\subset P(\omega)$ as follows. A point $x$ in the polar set is \emph{not} in $P^{1,0}(\omega)$ if around this point there is holomorphic function $g$, with $g\not\equiv0$, such that $\bar{g}\omega$ is smooth. In the same spirit we define the set $P^{0,1}(\omega)$ to be the subset of polar points around which there is \emph{not} a holomorphic function $f$, with $f\not\equiv0$, such that $f\omega$ is smooth. We say that $P^{1,0}(\omega)$ is the set where $\omega$ has holomorphic singularities and $P^{0,1}(\omega)$ is the set where $\omega$ has anti-holomorphic singularities. We have that
\[P(\omega)=P^{1,0}(\omega)\cup P^{0,1}(\omega)\]
but $P^{1,0}(\omega)\cap P^{0,1}(\omega)$ need not be empty; it is the set where $\omega$ has both holomorphic and anti-holomorphic singularities. The order of the holomorphic (and anti-holomorphic) pole is locally well defined.

If $\omega\in\mathcal{E}(*\bar{*}D)$ then $P^{1,0}(\omega)$ and $P^{0,1}(\omega)$ are hypersurfaces contained in $D$ and we temporarily set $H(\omega)$ to be the codimension one components of $P^{1,0}(\omega)\cap P^{0,1}(\omega)$. Since this is an analytic set there is a natural stratification, see Proposition II.5.6 in \cite{Dem},
\begin{equation}\label{eqStratification}
H(\omega)_{d}\subset H(\omega)_{d-1}\subset\dots\subset H(\omega)_{1}\subset H(\omega)_{0}
\end{equation}
where
\renewcommand{\labelenumi}{(\roman{enumi})}
\begin{enumerate}
\item $H(\omega)_{0}=X$,
\item $H(\omega)_{1}=H(\omega)$,
\item if $k=2,\dots,d$ then $H(\omega)_{k}$ is $\big(H(\omega)_{k-1}\big)_{\mathrm{sing}}$ together with all the components of $H(\omega)_{k-1}$ with codimension greater than or equal to $k$.
\end{enumerate}

Notice that $H(\omega)_{k}\setminus H(\omega)_{k+1}$ is a $(d-k)$-dimensional complex manifold which is possibly empty.

\begin{definition}
With the stratification as above we define the integer $\kappa(\omega)$ to be the largest number $k$ such that $H(\omega)_{k}$ is non-empty. We further let $E(\omega):=H(\omega)_{\kappa(\omega)}$.
\end{definition}
The integer $\kappa(\omega)$ in some sense measures how bad the singularities of $\omega$ are. By definition $E(\omega)$ is a complex submanifold of dimension $d-\kappa(\omega)$.
\begin{example}
To clarify these notions we give an example in $\C^{3}$ in the case of normal crossings. For
\[\omega=\frac{1}{z_{1}\bar{z}_{1}(z_{1}-1)z_{2}\bar{z}_{3}}\]
we have
\begin{align*}
P^{1,0}&=\{z_{1}=0\}\cup\{z_{1}=1\}\cup\{z_{2}=0\},\\
P^{0,1}&=\{z_{1}=0\}\cup\{z_{3}=0\}.
\end{align*}
Thus $P^{1,0}\cap P^{0,1}=\{z_{1}=0\}\cup\{z_{1}=1,z_{3}=0\}$ and hence $H(\omega)=\{z_{1}=0\}$. Since this is smooth we get that $\kappa(\omega)=1$ and $E(\omega)=\{z_{1}=0\}$.
\end{example}

For a semi-meromorphic form $\alpha$ we have $H(\alpha)=\varnothing$. Hence all components except $H(\alpha)_{0}=X$ in the stratification are empty. Thus $\kappa(\alpha)=0$ and $E(\alpha)=X$.

For a form $\omega\in\mathcal{E}(*\bar{*}D)$, where $D$ has normal crossings, there is a more explicit description of $\kappa(\omega)$. Around any point $x\in X$ there are local coordinates $(z_{1},\dots,z_{d})$ with $D$ given by $z_{1}z_{2}\cdots z_{k}=0$. Then there are multi-indices $J$ and $K$ so that $z^{J}\bar{z}^{K}\omega$ is smooth. Choosing $J$ and $K$ minimal we define
\[\kappa_{x}(\omega)=\#\{j:J_{j}\neq0\text{ and }K_{j}\neq0\}\]
and then
\[\kappa(\omega)=\max_{x\in X}\kappa_{x}(\omega).\]

Now suppose $s:X\rightarrow L$ is a holomorphic section such that $D=\{s=0\}$ has normal crossings and that $\omega\in\mathcal{E}(*\bar{*}D)$. Around any point $x\in X$ there are coordinates $(z_{1},\dots,z_{d})$ so that $H(\omega)$ is given by $z_{1}z_{2}\cdots z_{\ell}=0$. In a local holomorphic frame the section is given by $s=z^{I}\phi$ for some holomorphic $\phi$ which is non-vanishing on $H(\omega)$. We define
\begin{equation}\label{eqOrder}
o_{\omega,x}(s)=\prod_{j=1}^{\ell}I_{j}.
\end{equation}
and note that this does not depend on the choices of local coordinates or the frame.

\begin{definition}
For a holomorphic section $s:X\rightarrow L$ which defines a hypersurface $D$ with normal crossings  and $\omega\in\mathcal{E}(*\bar{*}D)$ we let
\[o_{\omega}(s)=\max_{x\in X}o_{\omega,x}(s).\qedhere\]
\end{definition}

Notice that in (\ref{eqOrder}) we only multiply with the vanishing order for $s$ on the local components on which $\omega$ has both holomorphic and anti-holomorphic poles. For $\omega$ semi-meromorphic $o_{\omega}(s)=1$ for all sections $s$ since then the product is empty.

We are now assuming that the polar set of $\omega$ is a hypersurface with normal crossings. For a test form $\xi$ of complementary degree and $\lambda\in\C$ with $\mathrm{Re}(\lambda)\gg1$ we let
\begin{equation}\label{defF}
F_{\xi}(\lambda)=o_{\omega}(s)\int_{X}|s|^{2\lambda}\omega\wedge\xi.
\end{equation}

%\begin{remark}
%In the case that $\omega$ is semi-meromorphic the function $F_{\xi}$ is the function defined in the previous section since then $o_{\omega}(s)=1$.
%\end{remark}

The following theorem gives a first description of the function $F_{\xi}$.

\begin{theorem}\label{thmPrinc}
Suppose $\omega\in\QM(X)$ has a hypersurface $D$ with normal crossings as a polar set. The function $F_{\xi}$ has the following properties
\begin{description}
\item{(a)} $F_{\xi}$ has a meromorphic extension to $\C$,
\item{(b)} the possible poles of $F_{\xi}$ are at $\Q\subset\R$,
\item{(c)} the order of the pole of $F_{\xi}$ at the origin is $\leq\kappa(\omega)$.
\end{description}
\end{theorem}

To prove Theorem \ref{thmPrinc} we need the following lemma, the proof of which is a simple exercise.

\begin{lemma}\label{lemmaMulti}
For $\lambda\in\C$ and multi-indices $I,J,K$ such that if $I_{j}=0$ then $J_{j}=0$ and $K_{j}=0$ we have
\[\frac{|z^{I}|^{2\lambda}}{z^{J}\bar{z}^{K}}=\frac{h(\lambda)}{\lambda^{p}}\frac{\partial^{J+K}|z^{I}|^{2\lambda}}{\partial z^{J}\partial\bar{z}^{K}}\]
where
\[h(\lambda)=\Big(\prod_{J_{j}\neq0}I_{j}(\lambda I_{j}-1)\cdots(\lambda I_{j}-J_{j}+1)\Big)^{-1}\Big(\prod_{K_{j}\neq0}I_{j}(\lambda I_{j}-1)\cdots(\lambda I_{j}-K_{j}+1)\Big)^{-1}\]
and $p=\#\{j:J_{j}\neq0\}+\#\{j:K_{j}\neq0\}$.
\end{lemma}

Notice that this means that $h(\lambda)$ has poles in
\[\lambda=\frac{1}{I_{j}},\frac{2}{I_{j}},\dots,\frac{J_{j}-1}{I_{j}}\quad\text{for }j\text{ with } J_{j}>1\]
and
\[\lambda=\frac{1}{I_{j}},\frac{2}{I_{j}},\dots,\frac{K_{j}-1}{I_{j}}\quad\text{for }j\text{ with } K_{j}>1.\]

%\begin{proof}
%Beginning with the derivatives with respect to $z_{1}$ on the right hand side we have
%\begin{align*}
%\frac{\partial^{J_{1}}(z^{I}\bar{z}^{I})^{\lambda}}{\partial z_{1}^{J_{1}}}&=|z_{2}^{I_{2}}\dots z_{d}^{I_{d}}|^{2\lambda}\frac{\partial^{J_{1}}}{\partial z_{1}^{J_{1}}}\Big((z_{1}\bar{z}_{1})^{\lambda I_{1}}\Big)\\
%&=(\lambda I_{1})\bar{z}_{1}|z_{2}^{I_{2}}\dots z_{d}^{I_{d}}|^{2\lambda}\frac{\partial^{J_{1}-1}}{\partial z_{1}^{J_{1}-1}}\Big((z_{1}\bar{z}_{1})^{\lambda I_{1}-1}\Big)\\
%&=(\lambda I_{1})\dots(\lambda I_{1}-J_{1}+1)\bar{z}_{1}^{J_{1}}|z_{2}^{I_{2}}\dots z_{d}^{I_{d}}|^{2\lambda}(z_{1}\bar{z}_{1})^{\lambda I_{1}-J_{1}}\\
%&=(\lambda I_{1})\dots(\lambda I_{1}-J_{1}+1)|z^{I}|^{2\lambda}z_{1}^{-J_{1}}.
%\end{align*}
%Continuing with the other derivatives gives the result.
%\end{proof}

\begin{proof}[Proof of Theorem \ref{thmPrinc}]
We may suppose that $\xi$ has support in a coordinate chart and so we study the integral over, say, a polydisc $\Delta\subset\C^{d}$. Since $D$ has normal crossings we may find coordinates so that the section $s$ is a monomial, say $s=z^{I}=z_{1}^{I_{1}}\cdots z_{d}^{I_{d}}$ and we write the metric as $|\cdot|=|\cdot|\e^{-\phi}$ for some function $\phi$. Furthermore, we write
\[\omega\wedge\xi=\frac{\psi}{z^{J}\bar{z}^{K}}\d z\wedge\d\bar{z}\]
where $\d z=\d z_{1}\wedge\dots\wedge\d z_{d}$ and $\psi$ is some smooth function with support in $\Delta$. The integral in (\ref{defF}) may now be written
\begin{equation}\label{defF2}
F_{\xi}(\lambda)=o_{\omega}(s)\int_{\Delta}\frac{|z^{I}|^{2\lambda}}{z^{J}\bar{z}^{K}}\e^{-2\lambda\phi}\psi\,\d z\wedge\d\bar{z}.
\end{equation}
We now prove (a). For integers $N\geq0$ we can use Lemma \ref{lemmaMulti} and Stokes' theorem to simplify the integral in (\ref{defF2}) as
\begin{align*}
F_{\xi}(\lambda)&=o_{\omega}(s)\int_{\Delta}\frac{|z^{I}|^{2\lambda+2N}}{z^{J+NI}\bar{z}^{K+NI}}\e^{-2\lambda\phi}\psi\,\d z\wedge\d\bar{z}\\
&=\frac{o_{\omega}(s)h(\lambda)}{\lambda^{p_{N}}}\int_{\Delta}\frac{\partial^{J+K+2NI}|z^{I}|^{2\lambda+2N}}{\partial z^{J+NI}\partial\bar{z}^{K+NI}}\e^{-2\lambda\phi}\psi\,\d z\wedge\d\bar{z}\\
&=\frac{(-1)^{|J+NI|+|K+NI|}o_{\omega}(s)h(\lambda)}{\lambda^{p_{N}}}\int_{\Delta}|z^{I}|^{2\lambda+2N}\frac{\partial^{J+K+2NI}}{\partial z^{J+NI}\partial\bar{z}^{K+NI}}\big(\e^{-2\lambda\phi}\psi\big)\d z\wedge\d\bar{z}.
\end{align*}
The last integral in the above expression is holomorphic in $\mathrm{Re}\,\lambda>-N-\epsilon$ for some $\epsilon>0$. Furthermore, the function $h$, which is given by Lemma \ref{lemmaMulti} but here depends on $N$, is meromorphic in $\C$. Hence $F_{\xi}$ has a meromorphic extension to $\C$, as $N$ may be chosen arbitrarily large, and we have proven (a).

Now let us prove (b). The fact that the poles are located at rational numbers follows from the proof of (a) and Lemma \ref{lemmaMulti} which describes the locations of the poles of $h$.
%To prove the remaining part of (b) we choose $N=0$ and let $h=h_{0}$ and $p=p_{0}$. The poles in the right half plane are now given by $h$ and they are located at
%\[\frac{1}{I_{i}},\dots,\frac{J_{i}-1}{I_{i}}\text{ and }\frac{1}{I_{i}},\dots,\frac{K_{i}-1}{I_{i}}\]
%for $i=1,\dots,d$.

Finally we prove (c). Choosing $N=0$ gives
\begin{align}\label{eqF}
F_{\xi}(\lambda)&=\frac{(-1)^{|J|+|K|}o_{\omega}(s)h(\lambda)}{\lambda^{p}}\int_{\Delta}|z^{I}|^{2\lambda}\frac{\partial^{J+K}}{\partial z^{J}\partial\bar{z}^{K}}\big(\e^{-2\lambda\phi}\psi\big)\d z\wedge\d\bar{z}.
\end{align}
Notice that Lemma \ref{lemmaMulti} in particular gives that $h$ does not have a pole at $0$. We define a function $g$ from the integral above by
\[g(\lambda)=\int_{\Delta}|z^{I}|^{2\lambda}\frac{\partial^{J+K}}{\partial z^{J}\partial\bar{z}^{K}}\big(\e^{-2\lambda\phi}\psi\big)\d z\wedge\d\bar{z}.\]
Then $g$ is holomorphic in $\mathrm{Re}\,\lambda>-\epsilon$ for some $\epsilon$. To show that $F_{\xi}$ has a pole of order $\kappa$ we need to show that $g$ has a zero of order $p-\kappa$ at the origin. We have that
\[p-\kappa=\#\{j:J_{j}\neq0\text{ or }K_{j}\neq0\}=\#\{j:I_{j}\neq0\}.\]
Repeated use of the product rule for derivatives gives
\begin{equation}\label{egG}
g^{(k)}(0)=\sum_{\ell=0}^{k}{{k}\choose{\ell}}(-2)^{k-\ell}\int_{\Delta}\big(\log|z^{I}|^{2}\big)^{\ell}\frac{\partial^{J+K}}{\partial z^{J}\partial\bar{z}^{K}}\big(\psi\phi^{k-\ell}\big)\d z\wedge\d\bar{z}
\end{equation}
and using the multinomial theorem we get
\begin{align}\label{egG2}
&\int_{\Delta}\big(\log|z^{I}|^{2}\big)^{\ell}\frac{\partial^{J+K}}{\partial z^{J}\partial\bar{z}^{K}}\big(\psi\phi^{k-\ell}\big)\d z\wedge\d\bar{z}\nonumber\\
&=\sum_{M}{{\ell}\choose{M}}\int_{\Delta}\prod_{j=1}^{d}\big(I_{j}\log|z_{j}|^{2}\big)^{M_{j}}\frac{\partial^{J+K}}{\partial z^{J}\partial\bar{z}^{K}}\big(\psi\phi^{k-\ell}\big)\d z\wedge\d\bar{z}.
\end{align}
The sum is over multi-indices $M=(M_{1},\dots,M_{d})$ such that $I_{j}=0$ implies that $M_{j}=0$, all $M_{j}\geq0$ and $\sum_{j}M_{j}=\ell$. Thus we have to study integrals of the form
\begin{align}\label{eqInt2}
\int_{\Delta}\prod_{j=1}^{d}\big(I_{j}\log|z_{j}|^{2}\big)^{M_{j}}\frac{\partial^{J+K}}{\partial z^{J}\partial\bar{z}^{K}}\big(\psi\phi^{k-\ell}\big)\d z\wedge\d\bar{z}.
\end{align}
Suppose first that $I_{1}\neq0$ but $M_{1}=0$. Then the integral in (\ref{eqInt2}) may be written
\begin{align*}
\int_{\Delta'}\prod_{j=2}^{d}\big(I_{j}\log|z_{j}|^{2}\big)^{M_{j}}\bigg(\int_{\Delta_{1}}\frac{\partial^{J+K}}{\partial z^{J}\partial\bar{z}^{K}}\big(\psi\phi^{k-\ell}\big)\d z_{1}\wedge\d\bar{z}_{1}\bigg)\d z'\wedge\d\bar{z}'
\end{align*}
where $\Delta=\Delta_{1}\times\Delta'$. But since $I_{j}\neq0$ implies that $J_{1}\neq0$ or $K_{1}\neq0$ the inner integral vanishes using Stokes' theorem. Hence we get the following: \[\text{\emph{if $I_{j}\neq0$ but $M_{j}=0$ then the integral in (\ref{eqInt2}) vanishes.}}\]
Now we suppose $k<p-\kappa$ and we want to show that $g^{(k)}(0)=0$. From (\ref{egG}) and (\ref{egG2}) we know that $g^{(k)}(0)$ is a sum of integrals as in (\ref{eqInt2}). For each of these integrals there are an integer $\ell$ and a multi-index $M$ such that
\[\sum M_{j}=\ell<p-\kappa=\#\{j:I_{j}\neq0\}.\]
Hence, for each of the integrals, there is some $j$ so that $I_{j}\neq0$ but $M_{j}=0$. Then, as explained above, all of the integrals are zero and thus $g^{(k)}(0)=0$ for $k<p-\kappa$. Therefore $g$ has a zero of order $p-\kappa$ at the origin which was what we wanted to prove.
\end{proof}

We use Theorem \ref{thmPrinc} (c) to make the following definition.

\begin{definition}\label{defCanonicalCurrent}
For $\omega\in\E(*\bar{*}D)$, where $D$ has normal crossings, we define the \emph{canonical current $\{\omega\}$ associated to $\omega$} by
\[\big\langle\{\omega\},\xi\big\rangle=\lambda^{\kappa(\omega)}F_{\xi}(\lambda)\Big|_{\lambda=0}.\qedhere\]
\end{definition}

A priori $\{\omega\}$ depends on choice of $s$ and $|\cdot|$. Corollary \ref{corIndep}, however, shows that this is not the case.

\begin{remark}
In the case that $\omega$ is semi-meromorphic $\{\omega\}$ is the principal value current of $\omega$ since then $\kappa(\omega)=0$ and $o_{\omega}(s)=1$.
\end{remark}

%%%%%%%%%%%%%%%%%%%%%%
\subsection{Local calculations} 
%%%%%%%%%%%%%%%%%%%%%%

We will make some calculations of canonical currents associated to quasi-meromorphic forms to hopefully clarify but also to show that they can behave a bit odd. Given a multi-index $J=(J_{1},\dots,J_{d})$ we write $1_{J}$ for the multi-index given by $(1_{J})_{j}=0$ if $J_{j}=0$ and $(1_{J})_{j}=1$ if $J_{j}\neq0$. We begin with a proposition.

\begin{proposition}\label{PropLocal}
For $\omega\in\QM(\C^{d})$ and a test function $\xi$ in $\C^{d}$ with support in $\Delta$ such that $\omega\wedge\xi=(\psi/z^{J}\bar{z}^{K})\d z\wedge\d\bar{z}$ we have
\[\big\langle\{\omega\},\xi\big\rangle=\frac{(-1)^{p}}{(J-1_{J})!(K-1_{K})!}\int_{\Delta}\Big(\prod_{j:J_{j}+K_{j}\neq0}\log|z_{j}|^{2}\Big)\frac{\partial^{J+K}\psi}{\partial z^{J}\partial\bar{z}^{K}}\d z\wedge\d\bar{z}\]
where $p$ is given by Lemma \ref{lemmaMulti}.
\end{proposition}
\begin{proof}
From the proof of Theorem \ref{thmPrinc} we know
\[\big\langle\{\omega\},\xi\big\rangle=\lambda^{\kappa(\omega)}F_{\xi}(\lambda)\Big|_{\lambda=0}=\frac{o_{\omega}(s)(-1)^{|J|+|K|}}{(p-\kappa(\omega))!}h(0)g^{(p-\kappa(\omega))}(0)\]
and Lemma \ref{lemmaMulti} gives 
\[h(0)=\frac{(-1)^{|J|+|K|-p}}{(J-1_{J})!(K-1_{K})!}\Big(\prod_{j:J_{j}\neq0}I_{j}\Big)^{-1}\Big(\prod_{j:K_{j}\neq0}I_{j}\Big)^{-1}.\]
The equation (\ref{egG}) gives an expression for $g^{(p-\kappa(\omega))}(0)$
in terms of the integrals in (\ref{egG2}). But just as in the proof of Theorem \ref{thmPrinc} these integrals vanish if $\ell<p-\kappa(\omega)$. For $\ell=p-\kappa(\omega)$ we must have all $M_{j}=1$ for the integral not to vanish. Using this for $k=p-\kappa(\omega)$ we get
\[g^{(p-\kappa(\omega))}(0)=\Big(\prod_{j:I_{j}\neq0}I_{j}\Big)(p-\kappa(\omega))!\int_{\Delta}\Big(\prod_{j:I_{j}\neq0}\log|z_{j}|^{2}\Big)\frac{\partial^{J+K}\psi}{\partial z^{J}\partial\bar{z}^{K}}\d z\wedge\d\bar{z}.\]
This is the same integral as in the statement of the proposition. We only need to see what constant we get in front of it. This constant is
\begin{align*}
o_{\omega}(s)\frac{(-1)^{p}}{(J-1_{J})!(K-1_{K})!}\Big(\prod_{j:I_{j}\neq0}I_{j}\Big)\Big(\prod_{j:J_{j}\neq0}I_{j}\Big)^{-1}\Big(\prod_{j:K_{j}\neq0}I_{j}\Big)^{-1}
\end{align*}
but since $o_{\omega}(s)=\prod_{j:J_{j}\neq0,K_{j}\neq0}I_{j}$ this is precisely what is claimed.
\end{proof}

\begin{corollary}\label{corIndep}
The canonical current $\{\omega\}$ does not depend on the choice of section $s$ or metric $|\cdot|$.
\end{corollary}
\begin{proof}
This follows immediately from Proposition \ref{PropLocal} since the right hand side in that statement does not depend on the section $s$ or the metric $|\cdot|$, as $J$ and $K$ do not. Hence (locally and thus also globally) this holds for $\{\omega\}$.
\end{proof}

\begin{remark}
We would not get the above corollary if we did not have the factor $o_{\omega}(s)$ in the definition of $F_{\xi}$.
\end{remark}

When doing calculations we will get use of the following which is a consequence of Cauchy--Green's theorem: If $\psi$ is a smooth function with compact support in $\Delta\subset\C$ then
\begin{equation}\label{goodEq}
\psi(0)=-\frac{1}{2\pi\i}\int_{\Delta}\log|z|^{2}\frac{\partial^{2}\psi}{\partial z\partial\bar{z}}\d z\wedge\d\bar{z}.
\end{equation}

\begin{corollary}\label{corLocal}
For $\omega\in\QM(\C^{d})$ and a test function $\xi$ in $\C^{d}$ with support in $\Delta$ we have
\begin{description}
\item{(a)} if $\omega\wedge\xi=(\psi/z_{1}^{m}\bar{z}_{1}^{n})\d z\wedge\d\bar{z}$ then
\[\big\langle\{\omega\},\xi\big\rangle=-\frac{2\pi\i}{(m-1)!(n-1)!}\int_{\Delta\cap\{z_{1}=0\}}
\frac{\partial^{m+n-2}\psi}{\partial z_{1}^{m-1}\partial\bar{z}_{1}^{n-1}}\d z'\wedge\d\bar{z}',\]
\item{(b)} if $\omega\wedge\xi=(\psi/z_{1}^{J_{1}}\dots z_{k}^{J_{k}}\bar{z}_{1}\dots\bar{z}_{k})\d z\wedge\d\bar{z}$
\[\big\langle\{\omega\},\xi\big\rangle=\frac{(-2\pi\i)^{k}}{(J-1_{J})!}\int_{\Delta\cap\{z_{1}=\dots=z_{k}=0\}}\frac{\partial^{J-1_{J}}\psi}{\partial z^{J-1_{J}}}\d z''\wedge\d\bar{z}''\]
\end{description}
where $\d z'\wedge\d\bar{z}'=\d z_{2}\wedge\d\bar{z}_{2}\wedge\dots\wedge\d z_{d}\wedge\d\bar{z}_{d}$ and $\d z''\wedge\d\bar{z}''=\d z_{k+1}\wedge\d\bar{z}_{k+1}\wedge\dots\wedge\d z_{d}\wedge\d\bar{z}_{d}$.
\end{corollary}
\begin{proof}
This follows from Proposition \ref{PropLocal} and (\ref{goodEq}).
\end{proof}

We now use Corollary \ref{corLocal} to make some explicit calculations.

\begin{example}
Let $X=\CP^{1}$ with homogeneous coordinates $[z:w]$ and let $0$ be the point where $z=0$ and $\infty$ the point where $w=0$. We let
\[\omega=\frac{\d z\wedge\d\bar{z}}{z\bar{z}}=\frac{\d w\wedge\d\bar{w}}{w\bar{w}}\quad\text{for }zw\neq0,\]
which means that $\kappa(\omega)=1$. In view of Corollary \ref{corLocal} (a), given a test function $\xi$, we get
\begin{align*}
\big\langle\{\omega\},\xi\big\rangle&=-2\pi\i\xi(0)-2\pi\i\xi(\infty).
\end{align*}
On the other hand, if $X=U$ for some open set $U\subset\CP^{1}$ which does not contain the origin or $\infty$ then $\kappa(\omega)=0$ and therefore
\begin{align*}
\big\langle\{\omega\},\xi\big\rangle&=\int_{U}\frac{\xi(z)}{|z|^{2}}\d z\wedge\d\bar{z}.\qedhere
\end{align*}
\end{example}

\begin{remark}
The above example shows that for canonical currents we have the following property: in general $\chi\{\omega\}\neq\{\chi\omega\}$ for a smooth function $\chi$. This means that when we define the canonical current associated to a form $\omega$ it is important to decide on what underlying space we consider it.
\end{remark}

\begin{example}
If we let $X=\C$ and apply Corollary \ref{corLocal} with $\omega=1/(z^{m}\bar{z}^{n})$ then we get that
\[z\bigg\{\frac{1}{z^{m}\bar{z}^{n}}\bigg\}=\bigg\{\frac{1}{z^{m-1}\bar{z}^{n}}\bigg\}\quad\text{and}\quad\bar{z}\bigg\{\frac{1}{z^{m}\bar{z}^{n}}\bigg\}=\bigg\{\frac{1}{z^{m}\bar{z}^{n-1}}\bigg\}\]
for $m,n\geq2$. On the other hand
\[z^{m}\bigg\{\frac{1}{z^{m}\bar{z}^{n}}\bigg\}=0\quad\text{and}\quad\bar{z}^{n}\bigg\{\frac{1}{z^{m}\bar{z}^{n}}\bigg\}=0\]
for $m,n\geq1$.
\end{example}

Theorem \ref{thmPrinc} (b) gives some insight about the poles of $F_{\xi}$ but the following proposition gives more information.

\begin{proposition}\label{propPoles}
The poles of the function $F_{\xi}$ are located at rational numbers less than or equal to
\[\max\bigg\{\min\Big\{\frac{J_{j}-1}{I_{j}},\frac{K_{j}-1}{I_{j}}\Big\}:j=1,\dots,d\bigg\}.\]
\end{proposition}
\begin{proof}
First suppose $K_{i}=0$ or $K_{i}=1$ for all $i=1,\dots,d$. We may assume that $\xi$ has support in a local chart and so we can write down the integral locally as
\begin{align*}
F_{\xi}(\lambda)&=o_{\omega}(s)\int_{\Delta}\frac{|z^{I}|^{2\lambda}}{z^{J}\bar{z}^{K}}\e^{-2\lambda\phi}\psi\d z\wedge\d\bar{z}\\
&=\frac{h(\lambda)}{\lambda^{p}}\int_{\Delta}\frac{\partial^{K}|z^{I}|^{2\lambda}}{\partial\bar{z}^{K}}\frac{1}{z^{J}}\e^{-2\lambda\phi}\psi\d z\wedge\d\bar{z}\\
&=\frac{(-1)^{|K|}h(\lambda)}{\lambda^{p}}\int_{\Delta}\frac{|z^{I}|^{2\lambda}}{z^{J}}\frac{\partial^{K}\e^{-2\lambda\phi}\psi}{\partial\bar{z}^{K}}\d z\wedge\d\bar{z}.
\end{align*}
We made a similar computation in the proof of Theorem \ref{thmPrinc}, cf.\! Lemma \ref{lemmaMulti}, but now we only considered the anti-holomorphic derivatives. Since these are of at most order one the function $h$ will not have any poles at all, see Lemma \ref{lemmaMulti}. But the integral in the last expression above is the principal value current of $1\big/z^{J}$ acting on $\frac{\partial^{K}\e^{-2\lambda\phi}\psi}{\partial\bar{z}^{K}}\d z\wedge\d\bar{z}$. This is known not to have any poles in the right half plane (and not at the origin). Hence $F_{\xi}$ does not have any poles in $\mathrm{Re}(\lambda)>0$.

Note that the above result would also hold as long as $J_{i}\leq1$ or $K_{i}\leq1$ for all $i$. Now suppose we are in the general case. Let $\mu=\lambda-M$ for some integer $M$. Then
\[\frac{|z^{I}|^{2\lambda}}{z^{J}\bar{z}^{K}}=\frac{|z^{I}|^{2\mu+2M}}{z^{J}\bar{z}^{K}}=|z^{I}|^{2\mu}\frac{z^{MI}\bar{z}^{MI}}{z^{J}\bar{z}^{K}}\]
and choosing $M$ so that $MI_{i}\geq J_{i}-1$ or $MI_{i}\geq K_{i}-1$ for each $i$ we get from the above that $F_{\xi}$ has no poles in $\mathrm{Re}(\mu)>0$. That is, $F_{\xi}$ has no poles in $\mathrm{Re}(\lambda)>M$. Choosing $M$ so that this holds we get the proposition.
\end{proof}

One can note that by choosing higher powers $I$ of the section $s$ we can get the poles in the right half-plane arbitrarily close to the origin. Suppose $\omega=\alpha\wedge\bar{\beta}$ for semi-meromorphic forms $\alpha$ and $\beta$. Proposition \ref{propPoles} gives us a hint that the situation is a bit more well behaved when $\beta$ only has poles of order one since then the proposition says that $F_{\xi}$ does not have poles in the right half plane.

%%%%%%%%%%%%%%%%%%%%%%
\section{Cohomological residues}
%%%%%%%%%%%%%%%%%%%%%%

We will discuss the classical Leray residue, the conjugate Dolbeault residue and then define a residue for the Aeppli cohomology. Now $X$ is assumed to be a compact complex manifold.

%%%%%%%%%%%%%%%%%%%%%%
\subsection{The conjugate Dolbeault residue}
%%%%%%%%%%%%%%%%%%%%%%

To define residues the classical setting is the following: suppose $D$ is a smooth hypersurface and $\alpha$ a $\d$-closed form in $X\setminus D$ with a holomorphic pole of order one along $D$. If $z_{1}=0$ is a local equation for $D$ then $\alpha$ may locally be written as
\[\alpha=\frac{\d z_{1}}{z_{1}}\wedge\widetilde{\alpha}+\tau\]
for some forms $\widetilde{\alpha}$ and $\tau$ such that $\tau$ does not contain $\d z_{1}$. Certainly $\widetilde{\alpha}$ is smooth but it is well known that the closedness implies that $\tau$ is smooth. One defines the \emph{Poincaré residue} by $\Res(\alpha)=\widetilde{\alpha}\big|_{D}$. It is easy to check that this gives a well defined closed form on $D$. If $\alpha$ is any closed form on $X\setminus D$ then there is a cohomologous form $\alpha'$ with a pole of order one along $D$, cf.\! \cite[Thm. 6.3.3, p.~233]{Ch}. The \emph{Leray residue} is defined by
\[\mathrm{Res}(\alpha)=\big[\Res(\alpha')\big]_{\mathrm{dR}}\]
which gives a map
\[\Res:H^{k}(X\setminus D)\rightarrow H^{k-1}(D).\]

Since the groups $\E^{p,q}(*D)$ form a complex with the operator $\partial$ we get cohomology groups $H^{p,q}_{\partial}(*D)$. In \cite{Fe} the \emph{conjugate Dolbeault residue} was constructed as a map
\[\Res_{\partial}:H^{p,0}_{\partial}(*D)\rightarrow H_{\partial}^{p-1,0}(D).\]
We will give an alternative definition for forms in $H^{d,q}_{\partial}(*D)$ which is quite explicit. Given a $(d,q)$-form $\alpha$ in $\C^{d}$, with coordinates $z=(z_{1},\dots,z_{d})$, which has a holomorphic pole along $z_{1}=0$ we may write
\begin{equation}
\alpha=\frac{\d z_{1}\wedge\widetilde{\alpha}_{z}}{z_{1}^{m}}
\end{equation}
for some smooth form $\widetilde{\alpha}_{z}$ which does not contain $\d z_{1}$. To define a residue we need the following lemma. We do not give the proof since it is very similar to the proof of Lemma~\ref{lemmaAeppli} below.

\begin{lemma}\label{lemmaConj}
Let $z$ and $w$ be coordinates in $\C^{d}$ such that $z_{1}/w_{1}$ is a non-vanishing holomorphic function and let $D=\{z_{1}=0\}$. Suppose $\alpha\in\E^{d,q}(*D)$ has compact support and write
\[\frac{\d z_{1}}{z_{1}^{m}}\wedge\widetilde{\alpha}_{z}(z)=\alpha=\frac{\d w_{1}}{w_{1}^{m}}\wedge\widetilde{\alpha}_{w}(w),\]
for some smooth forms $\widetilde{\alpha}_{z}(z)$ and $\widetilde{\alpha}_{w}(w)$ which does not contain $\d z_{1}$ or $\d w_{1}$.
\begin{description}
\item{(a)} If there is a form $\eta\in\E(*D)$ with compact support such that $\alpha=\partial\eta$ then there is a smooth form $\widehat{\eta}$ on $D$ such that
\[\frac{\partial^{m-1}\widetilde{\alpha}_{z}}{\partial z_{1}^{m-1}}\bigg|_{D}=\partial\widehat{\eta},\]
with $\supp(\widehat{\eta})\subset\supp(\alpha)\cap D$.
\item{(b)} There is a smooth form $\beta$ on $D$ whose support is contained in $\supp(\alpha)\cap D$ such that
\[\frac{\partial^{m-1}\widetilde{\alpha}_{z}}{\partial z_{1}^{m-1}}\bigg|_{D}=\frac{\partial^{m-1}\widetilde{\alpha}_{w}}{\partial w_{1}^{m-1}}\bigg|_{D}+\partial\beta.\]
\end{description}
\end{lemma}

Now suppose $\alpha\in\E^{d,q}(*D)$ and $(\rho_{j})$ is a partition of unity subordinate to a cover of $X$ by charts with coordinates $\big(z_{j}=(z_{j,1},z_{j,2}\dots,z_{j,d})\big)$ such that $D$ is locally given by $z_{j,1}=0$. We write
\[\alpha=\frac{\d z_{j,1}\wedge\widetilde{\alpha}_{j}(z)}{z_{j,1}^{m}}\quad\text{on }\supp(\rho_{j}),\]
and then define
\[R_{\rho,z}(\omega)=\sum_{j}\frac{1}{(m-1)!}\frac{\partial^{m-1}(\rho_{j}\widetilde{\alpha}_{j})}{\partial z_{j,1}^{m-1}}\bigg|_{D}.\]
Using Lemma \ref{lemmaConj} one can prove that, for $\alpha\in\E^{d,q}(*D)$,
\begin{description}
\item{(a)} $R_{\rho,z}(\alpha)=R_{\sigma,w}(\alpha)+\partial\beta$,
\item{(b)} $R_{\rho,z}(\partial\eta)=\partial\widehat{\eta}$.
\end{description}
The proof of (a) and (b) is very similar to the proof of Proposition \ref{propAeppli} below. We can now make the following definition.

\begin{definition}\label{definitionConj}
For a class $[\alpha]\in H^{d,q}_{\partial}(*D)$ we define its \emph{conjugate Dolbeault residue} by 
\[\Res_{\partial}(\alpha)=\big[R_{\rho,z}(\alpha)\big]_{\partial}.\qedhere\]
\end{definition}

The claims (a) and (b) above give that $\Res_{\partial}(\alpha)$ is well defined and independent of the choice of partition of unity and local coordinates. We now present a theorem which is not very related to the rest of the paper, but we think it is a nice application of the conjugate Dolbeault residue.

\begin{theorem}\label{thmResidue}
If $\alpha\in\E^{p,q}(*D)$, where $D$ is a smooth hypersurface, and $\xi$ a test form of bidegree $(d-p,d-q-1)$ then
\begin{align*}
\big\langle\dbar[\alpha],\xi\big\rangle&=\big\langle[\dbar\alpha],\xi\big\rangle+2\pi\i\int_{D}\Res_{\partial}(\alpha\wedge\xi).
\end{align*}
\end{theorem}

\begin{proof}
We may suppose $\xi$ has support contained in a coordinate chart which is biholomorphic to the unit polydisc $\Delta$ and that $D$ is there given by $z_{1}=0$. We may further suppose that $\alpha=\frac{a}{z_{1}^{m}}\d z_{P}\wedge\d\bar{z}_{Q}$ and $\xi=b\,\d z_{R}\wedge\d\bar{z}_{S}$ where $|P|=p$ and $|Q|=q$. Then we get
\begin{align*}
\alpha\wedge\xi&=(-1)^{q(d-p)+s}\frac{ab}{z_{1}^{m}}\d z\wedge\d\bar{z}_{Q}\wedge\d\bar{z}_{S},\\
\alpha\wedge\dbar\xi&=\sum_{k}(-1)^{(q+1)(d-p)+s+t}\frac{a}{z_{1}^{m}}\frac{\partial b}{\partial\bar{z}_{k}}\d z\wedge\d\bar{z},\\
\dbar\alpha\wedge\xi&=\sum_{k}(-1)^{q(d-p)+d+q+s+t}\frac{\partial a}{\partial\bar{z}_{k}}\frac{b}{z_{1}^{m}}\d z\wedge\d\bar{z},
\end{align*}
where $s$ and $t$ are given by $\d z_{P}\wedge\d z_{R}=(-1)^{s}\d z$ and $\d\bar{z}_{Q}\wedge\d\bar{z}_{k}\wedge\d\bar{z}_{S}=(-1)^{t}\d\bar{z}$ (so $t$ depends on $k$ but we suppress this). For $k=1$ we have
\[(-1)^{t}\d\bar{z}=\d\bar{z}_{Q}\wedge\d\bar{z}_{1}\wedge\d\bar{z}_{S}=(-1)^{q}\d\bar{z}_{1}\wedge\d\bar{z}_{Q}\wedge\d\bar{z}_{S}\]
and hence
\[\d\bar{z}':=\d\bar{z}_{2}\wedge\dots\wedge\d\bar{z}_{d}=(-1)^{q+t}\d\bar{z}_{Q}\wedge\d\bar{z}_{S}.\]
This means that
\[\Res_{\partial}(\alpha\wedge\xi)=\frac{(-1)^{q(d-p)+q+s+t}}{(m-1)!}\frac{\partial^{m-1}(ab)}{\partial z_{1}^{m-1}}\d z'\wedge\d\bar{z}'.\]
We write $\Delta'=\Delta\cap\{z_{1}=0\}=\Delta\cap D$. Using Proposition \ref{PropLocal} and the remark after Definition \ref{defCanonicalCurrent} we get
\begin{align*}
\big\langle\dbar[\alpha],\xi\big\rangle&=(-1)^{p+q+1}\big\langle[\alpha],\dbar\xi\big\rangle\\
&=\sum_{k}\frac{(-1)^{q(d-p)+q+d+s+t}}{(m-1)!}\int_{\Delta}\log|z_{1}|^{2}\frac{\partial^{m}}{\partial z_{1}^{m}}\Big(a\frac{\partial b}{\partial\bar{z}_{k}}\Big)\d z\wedge\d\bar{z}\\
&=\frac{(-1)^{q(d-p)+q+d+s+t}}{(m-1)!}\int_{\Delta}\log|z_{1}|^{2}\frac{\partial^{m+1}(ab)}{\partial z_{1}^{m}\partial\bar{z}_{1}}\d z\wedge\d\bar{z}\\
&-\sum_{k}\frac{(-1)^{q(d-p)+q+d+s+t}}{(m-1)!}\int_{\Delta}\log|z_{1}|^{2}\frac{\partial^{m}}{\partial z_{1}^{m}}\Big(\frac{\partial a}{\partial\bar{z}_{k}}b\Big)\d z\wedge\d\bar{z}\\
&=\frac{2\pi\i(-1)^{q(d-p)+q+s+t}}{(m-1)!}\int_{\Delta'}\frac{\partial^{m-1}(ab)}{\partial z_{1}^{m-1}}\d z'\wedge\d\bar{z}'+\big\langle[\dbar\alpha],\xi\big\rangle\\
&=2\pi\i\int_{D}\Res_{\partial}(\alpha\wedge\xi)+\big\langle[\dbar\alpha],\xi\big\rangle\qedhere
\end{align*}
\end{proof}

%%%%%%%%%%%%%%%%%%%%%%
\subsection{A residue for the Aeppli cohomology}
%%%%%%%%%%%%%%%%%%%%%%

Recall that for a complex manifold $X$ one defines the Bott--Chern cohomology groups by
\[H^{p,q}_{BC}(X)=\frac{\ker(\partial)\cap\ker(\bar{\partial})}{\im(\partial\dbar)}\]
and the Aeppli cohomology groups by
\[H^{p,q}_{A}(X)=\frac{\ker(\partial\bar{\partial})}{\im(\partial)+\im(\bar{\partial})}.\]
Given a hermitian metric on $X$ the induced Hodge star operator gives an isomorphism
\[*:H^{p,q}_{BC}(X)\rightarrow H^{n-p,n-q}_{A}(X)\]
so in this sense the Aeppli cohomology is dual to the Bott--Chern cohomology.
We have the following natural maps
\begin{center}\begin{tikzpicture}
	\matrix (m) [matrix of math nodes,row sep=3em,column sep=4em,minimum width=2em]
	{
     	& H^{p,q}_{BC}(X) &  \\
     	H^{p,q}_{\partial}(X) & H^{p+q}_{dR}(X) & H^{p,q}_{\dbar}(X) \\
	 & H^{p,q}_{A}(X) & \\};
  	\path[-stealth]
    		(m-1-2) 	edge (m-2-1)
    		(m-1-2)	edge (m-2-3)
    		(m-2-1)	edge (m-3-2)
		(m-2-3) 	edge (m-3-2)
		(m-1-2) 	edge (m-2-2)
		(m-2-2) 	edge (m-3-2);
\end{tikzpicture}\end{center}
and for a manifold on which the $\partial\dbar$-lemma holds all the outer maps are isomorphisms. In particular this is true for K\"{a}hler manifolds. For a more elaborate discussion on these facts we refer to \cite{Ang1,Ang2,Del}.

Restricting our attention to forms in $\E^{d,d}(*\bar{*}D)$ we consider the cohomology group $H^{d,d}_{A}(*\bar{*}D)$. To define a residue we need the following lemma.

\begin{lemma}\label{lemmaAeppli}
Let $z$ and $w$ be coordinates in $\C^{d}$ such that $z_{1}/w_{1}$ is a non-vanishing holomorphic function and let $D=\{z_{1}=0\}$. Suppose $\omega\in\E^{d,d}(*\bar{*}D)$ has compact support and write
\[\frac{\d z_{1}\wedge\d\bar{z}_{1}}{z_{1}^{m}\bar{z}_{1}^{n}}\wedge\widetilde{\omega}_{z}(z)=\omega=\frac{\d w_{1}\wedge\d\bar{w}_{1}}{w_{1}^{m'}\bar{w}_{1}^{n'}}\wedge\widetilde{\omega}_{w}(w),\]
for some smooth forms $\widetilde{\omega}_{z}(z)$ and $\widetilde{\omega}_{w}(w)$ which does not contain $\d z_{1},\d\bar{z}_{1}$ or $\d w_{1},\d\bar{w}_{1}$.
\begin{description}
\item{(a)} If there are forms $\eta,\nu\in\E(*\bar{*}D)$ with compact support such that $\omega=\partial\eta+\dbar\nu$ then there are smooth forms $\widehat{\eta}$ and $\widehat{\nu}$ on $D$ such that
\[\frac{\partial^{m+n-2}\widetilde{\omega}_{z}}{\partial z_{1}^{m-1}\partial\bar{z}_{1}^{n-1}}\bigg|_{D}=\partial\widehat{\eta}+\dbar\widehat{\nu},\]
with $\supp(\widehat{\eta}),\supp(\widehat{\nu})\subset\supp(\omega)\cap D$.
\item{(b)} There are smooth forms $\widehat{\alpha}$ and $\widehat{\beta}$ on $D$ whose support is contained in $\supp(\omega)\cap D$ such that
\[\frac{1}{(m-1)!(n-1)!}\frac{\partial^{m+n-2}\widetilde{\omega}_{z}}{\partial z_{1}^{m-1}\partial\bar{z}_{1}^{n-1}}\bigg|_{D}=\frac{1}{(m'-1)!(n'-1)!}\frac{\partial^{m+n-2}\widetilde{\omega}_{w}}{\partial w_{1}^{m'-1}\partial\bar{w}_{1}^{n'-1}}\bigg|_{D}+\partial\widehat{\alpha}+\dbar\widehat{\beta}.\]
\end{description}
\end{lemma}

\begin{proof}
We first prove (a) and suppose $\omega=\partial\eta$. If
\[\eta=\frac{\d z_{1}\wedge\d\bar{z}_{1}\wedge\eta_{1}+\d\bar{z}_{1}\wedge\eta_{2}}{z_{1}^{m-1}\bar{z}_{1}^{n}},\]
where $\eta_{1}$ and $\eta_{2}$ does not contain $\d z_{1}$ or $\d\bar{z}_{1}$,
then
\[\omega=\partial\eta=\frac{\d z_{1}\wedge\d\bar{z}_{1}}{z_{1}^{m}\bar{z}_{1}^{n}}\wedge\Big(-(m-1)\eta_{2}+z_{1}\partial\eta_{1}+z_{1}\frac{\partial\eta_{2}}{\partial z_{1}}\Big)\]
and therefore
\[\widetilde{\omega}_{z}=-(m-1)\eta_{2}+z_{1}\partial\eta_{1}+z_{1}\frac{\partial\eta_{2}}{\partial z_{1}}.\]
We get
\begin{align*}
\frac{\partial^{m+n-2}\widetilde{\omega}_{z}}{\partial z_{1}^{m-1}\partial\bar{z}_{1}^{n-1}}\bigg|_{D}&=\frac{\partial^{m+n-2}}{\partial z_{1}^{m-1}\partial\bar{z}_{1}^{n-1}}\Big(-(m-1)\eta_{2}+z_{1}\partial\eta_{1}+z_{1}\frac{\partial\eta_{2}}{\partial z_{1}}\Big)\bigg|_{D}\\
&=(m-1)\Big(-\frac{\partial^{m+n-2}\eta_{2}}{\partial z_{1}^{m-1}\partial\bar{z}_{1}^{n-1}}+\frac{\partial^{m+n-3}\partial\eta_{1}}{\partial z_{1}^{m-2}\partial\bar{z}_{1}^{n-1}}+\frac{\partial^{m+n-2}\eta_{2}}{\partial z_{1}^{m-1}\partial\bar{z}_{1}^{n-1}}\Big)\bigg|_{D}\\
&=\partial\Big((m-1)\frac{\partial^{m+n-3}\eta_{1}}{\partial z_{1}^{m-2}\partial\bar{z}_{1}^{n-1}}\Big|_{D}\Big).
\end{align*}
The case $\omega=\dbar\nu$ is treated analogously. By linearity we get the case $\omega=\partial\eta+\dbar\nu$ and hence we have proven (a). Now we prove (b) and we first suppose $(m,n)=(m',n')$. The calculation
\begin{align*}
\omega&=-\partial\Big(\frac{1}{m-1}\frac{\d\bar{z}_{1}\wedge\widetilde{\omega}_{z}}{z_{1}^{m-1}\bar{z}_{1}^{n}}\Big)-\frac{1}{m-1}\frac{\d\bar{z}_{1}\wedge\partial\widetilde{\omega}_{z}}{z_{1}^{m-1}\bar{z}_{1}^{n}}\\
&=-\partial\Big(\frac{1}{m-1}\frac{\d\bar{z}_{1}\wedge\widetilde{\omega}_{z}}{z_{1}^{m-1}\bar{z}_{1}^{n}}\Big)+\frac{1}{m-1}\frac{\d z_{1}\wedge\d\bar{z}_{1}}{z_{1}^{m-1}\bar{z}_{1}^{n}}\wedge\frac{\partial\widetilde{\omega}_{z}}{\partial z_{1}}
\end{align*}
may be iterated and so we can write
\[\omega=\partial\alpha_{1}+\dbar\beta_{1}+\frac{1}{(m-1)!(n-1)!}\frac{\d z_{1}\wedge\d\bar{z}_{1}}{z_{1}\bar{z}_{1}}\wedge\frac{\partial^{m+n-2}\widetilde{\omega}_{z}}{\partial z_{1}^{m-1}\partial\bar{z}_{1}^{n-1}}.\]
Doing the same for the coordinate $w$ we get that
\[\frac{\d z_{1}\wedge\d\bar{z}_{1}}{z_{1}\bar{z}_{1}}\wedge\frac{\partial^{m+n-2}\widetilde{\omega}_{z}}{\partial z_{1}^{m-1}\partial\bar{z}_{1}^{n-1}}-\frac{\d w_{1}\wedge\d\bar{w}_{1}}{w_{1}\bar{w}_{1}}\wedge\frac{\partial^{m+n-2}\widetilde{\omega}_{w}}{\partial w_{1}^{m-1}\partial\bar{w}_{1}^{n-1}}=\partial\alpha+\dbar\beta\]
for some $\alpha$ and $\beta$. Using (a) we get
\[\frac{\partial^{m+n-2}\widetilde{\omega}_{z}}{\partial z_{1}^{m-1}\partial\bar{z}_{1}^{n-1}}\bigg|_{D}=\frac{\partial^{m+n-2}\widetilde{\omega}_{w}}{\partial w_{1}^{m-1}\partial\bar{w}_{1}^{n-1}}\bigg|_{D}+\partial\widehat{\alpha}+\dbar\widehat{\beta}\]
which is what was to be proven. Now we treat the case that $(m,n)\neq(m',n')$ and for simplicity we suppose $m'\geq m$ and $n'\geq n$. We get
\begin{align*}
&\frac{1}{(m'-1)!(n'-1)!}\frac{\partial^{m'+n'-2}z_{1}^{m'-m}\bar{z}_{1}^{n'-n}\widetilde{\omega}_{z}}{\partial z_{1}^{m'-1}\partial\bar{z}_{1}^{n'-1}}\bigg|_{D}\\
&=\frac{1}{(m'-1)!(n'-1)!}{{m'-1}\choose{m'-m}}{{n'-1}\choose{n'-n}}\frac{\partial^{m'-m}z_{1}^{m'-m}}{\partial z_{1}^{m'-m}}\frac{\partial^{n'-n}\bar{z}_{1}^{n'-n}}{\partial\bar{z}_{1}^{n'-n}}\frac{\partial^{m+n-2}\widetilde{\omega}_{z}}{\partial z_{1}^{m-1}\partial\bar{z}_{1}^{n-1}}\bigg|_{D}\\
&=\frac{1}{(m-1)!(n-1)!}\frac{\partial^{m+n-2}\widetilde{\omega}_{z}}{\partial z_{1}^{m-1}\partial\bar{z}_{1}^{n-1}}\bigg|_{D}
\end{align*}
since the restriction to $D$ forces the correct amount of derivatives to land on $z_{1}^{m'-m}$ and $\bar{z}_{1}^{n'-n}$. This proves (b).
\end{proof}

For a form $\omega\in\E^{d,d}(*\bar{*}D)$ and a partition of unity $(\rho_{j})$ subordinate to a cover of $X$ by charts with coordinates $\big(z_{j}=(z_{j,1},z_{j,2}\dots,z_{j,d})\big)$ such that $D$ is locally given by $z_{j,1}=0$ and
\[\omega=\frac{\d z_{j,1}\wedge\d\bar{z}_{j,1}}{z_{j,1}^{m}\bar{z}_{j,1}^{n}}\wedge\widetilde{\omega}_{j}(z)\quad\text{on }\supp(\rho_{j}),\]
we let
\[\Res_{\rho,z}(\omega)=\sum_{j}\frac{1}{(m-1)!(n-1)!}\frac{\partial^{m+n-2}(\rho_{j}\widetilde{\omega}_{j})}{\partial z_{j,1}^{m-1}\partial\bar{z}_{j,1}^{n-1}}\bigg|_{D}.\qedhere\]

\begin{proposition}\label{propAeppli}
For $\omega\in\E^{d,d}(*\bar{*}D)$ we have
\begin{description}
\item{(a)} $\Res_{\rho,z}(\omega)=\Res_{\sigma,w}(\omega)+\partial\alpha+\dbar\beta$,
\item{(b)} $\Res_{\rho,z}(\partial\eta+\dbar\nu)=\partial\alpha+\dbar\beta$.
\end{description}
\end{proposition}

\begin{proof}
We write
\[\Res^{j}_{\rho,z}(\omega)=\frac{1}{(m-1)!(n-1)!}\frac{\partial^{m+n-2}(\rho_{j}\widetilde{\omega}_{j})}{\partial z_{j,1}^{m-1}\partial\bar{z}_{j,1}^{n-1}}\bigg|_{D}\]
so that
\[\Res_{\rho,z}(\omega)=\sum_{j}\Res^{j}_{\rho,z}(\omega).\]
We have the following two identities:
\begin{description}
\item{(i)} $\Res^{j}_{\rho,z}(\sigma_{i}\omega)=\Res^{i}_{\sigma,w}(\rho_{j}\omega)+\partial\alpha_{i,j}+\dbar\beta_{i,j}$,
\item{(ii)} $\Res^{j}_{\rho,z}(\omega)=\sum_{i}\Res^{j}_{\rho,z}(\sigma_{i}\omega)$.
\end{description}
The first is basically Lemma 3.5 (b) and (ii) is just an interchange of the differentiation and the sum. Using the claims we get
\begin{align*}
\Res_{\rho,z}(\omega)&\stackrel{\text{def}}{=}\sum_{j}\Res^{j}_{\rho,z}(\omega)\\
&\stackrel{(ii)}{=}\sum_{j,i}\Res^{j}_{\rho,z}(\sigma_{i}\omega)\\
&\stackrel{(i)}{=}\sum_{i,j}\Res^{i}_{\sigma,w}(\rho_{j}\omega)+\partial\alpha_{i,j}+\dbar\beta_{i,j}\\
&\stackrel{(ii)}{=}\sum_{i}\Res^{i}_{\sigma,w}(\omega)+\sum_{i,j}\partial\alpha_{i,j}+\dbar\beta_{i,j}\\
&\stackrel{\text{def}}{=}\Res_{\sigma,w}(\omega)+\partial\Big(\sum_{i,j}\alpha_{i,j}\Big)+\dbar\Big(\sum_{i,j}\beta_{i,j}\Big)
\end{align*}
since $\alpha_{i,j}$ and $\beta_{i,j}$ has support contained in $\mathrm{supp}(\rho_{j}\sigma_{i})$. Thus we have proven (a). We further have
\begin{align*}
\Res_{\rho,z}(\partial\eta+\dbar\nu)&=\Res_{\rho,z}\Big(\sum_{i}\partial(\sigma_{i}\eta)+\dbar(\sigma_{i}\nu)\Big)\\
&=\sum_{i}\Res_{\rho,z}\big(\partial(\sigma_{i}\eta)+\dbar(\sigma_{i}\nu)\big)\\
&=\sum_{i}\partial\alpha_{i}\\
&=\partial\Big(\sum_{i}\alpha_{i}\Big).
\end{align*}
which proves (b).
\end{proof}

Using Proposition \ref{propAeppli} we can give the following definition.

\begin{definition}\label{defAeppli}
For $\omega\in H_{A}^{d,d}(*\bar{*}D)$ we define the \emph{Aeppli residue} by
\[\Res_{A}(\omega)=[\Res_{\rho,z}(\omega)]_{A}\qedhere\]
\end{definition}

\begin{remark}
Our definition of the Aeppli residue is very similar to the definition of \emph{the residue map} in \cite{Fe2}. They define this in a different context and for forms with, what they call, tame singularities.
\end{remark}

We thus have a map $\Res_{A}:H^{d,d}_{A}(*\bar{*}D)\rightarrow H_{A}^{d-1,d-1}(D)$.

\begin{proposition}\label{AeppliProperties}
\begin{description}
\item{(a)} If $\omega\in H^{d,d}_{A}(*\bar{*}D)$ is semi-meromorphic then $\Res_{A}(\omega)=0$.
\item{(b)} If $\alpha$ and $\beta$ are meromorphic $(d,0)$-forms with poles along a smooth hypersurface $D$ and the pole of $\beta$ is of order one then
\[\Res_{A}(\alpha\wedge\bar{\beta})=(-1)^{d-1}\big[\Res_{\partial}\,\alpha\wedge\overline{\Res\,\beta}\big]_{A}\]
where the right hand side is a well defined class and $\Res\,\beta$ denotes the Poincaré residue.
\end{description}
\end{proposition}

\begin{proof}
We get (a) from Lemma \ref{lemmaAeppli} since we may choose $n\geq1$. To prove (b) write locally $\alpha=(a/z_{1}^{m})\d z$ and $\beta=(b/z_{1})\d z$. Then $\alpha\wedge\bar{\beta}=(-1)^{d-1}\big(a\bar{b}/(z_{1}^{m}\bar{z}_{1})\big)\d z_{1}\wedge\d\bar{z}_{1}\wedge\d z'\wedge\d\bar{z}'$ and hence
\[\Res_{A}(\alpha\wedge\bar{\beta})=(-1)^{d-1}\Big[\frac{\partial^{m-1}a}{\partial z_{1}^{m-1}}\bar{b}\,\d z'\wedge\d\bar{z}'\Big]_{A}\]
and $\Res_{\partial}(\alpha)=\big[\frac{\partial^{m-1}a}{\partial z_{1}^{m-1}}\d z'\big]_{\partial}$. The Poincaré residue $\Res\,\beta$ is meromorphic since $\beta$ is. Letting $R=\frac{\partial^{m-1}a}{\partial z_{1}^{m-1}}\d z'$ we get that $(-1)^{d}R\wedge\overline{\Res\,\beta}$ is a representative of $\Res_{A}(\alpha\wedge\bar{\beta})$ and $R$ is a representative of $\Res_{\partial}(\alpha)$. If we choose a different representative, say $R+\partial\gamma$, of $\Res_{\partial}(\alpha)$ we get
\begin{align*}
(R+\partial\gamma)\wedge\overline{\Res\,\beta}=R\wedge\overline{\Res\,\beta}+\partial(\gamma\wedge\overline{\Res\,\beta})
\end{align*}
and therefore $\big[\Res_{\partial}\,\alpha\wedge\overline{\Res\,\beta}\big]_{A}$ is well defined.
\end{proof}

The next theorem relates the Aeppli residue to the canonical currents defined in Section 2.2. It gives an indication that canonical currents do not behave like principle value currents but rather as residue currents.

\begin{theorem}\label{thmAeppli}
For $\omega\in\E(*\bar{*}D)$ with $\kappa(\omega)>0$ and $D$ a smooth hypersurface we have
\[\big\langle\{\omega\},\xi\big\rangle=-2\pi\i\int_{D}\Res_{A}(\omega\wedge\xi).\]
\end{theorem}

\begin{proof}
Choose a partition of unity $(\rho_{\iota})$ subordinate to a cover consisting of charts which are mapped to the unit polydisc in which the hypersurface is given by $z_{1}=0$. Suppose the holomorphic pole has order $m$ and the anti-holomorphic pole order $n$. Since $\kappa(\omega)>0$ by assumption we have $m,n>0$. Notice that $\kappa(\omega)>0$ together with that $D$ is smooth implies that $\kappa(\omega)=1$. Write locally $\omega\wedge\xi=\psi/(z_{1}^{m}\bar{z}_{1}^{n})\d z\wedge\d\bar{z}$. Then, using Proposition \ref{PropLocal}, (\ref{goodEq}) and Definition \ref{defAeppli} we get
\begin{align*}
\big\langle\{\omega\},\xi\big\rangle&=\sum_{\iota}\frac{1}{(m-1)!(n-1)!}\int_{\Delta}\log|z_{1}|^{2}\frac{\partial^{m+n}\rho_{\iota}\psi}{\partial z_{1}^{m}\partial\bar{z}_{1}^{n}}\d z\wedge\d\bar{z}\\
&=-2\pi\i\sum_{\iota}\frac{1}{(m-1)!(n-1)!}\int_{\Delta\cap D}\frac{\partial^{m+n-2}\rho_{\iota}\psi}{\partial z_{1}^{m-1}\partial\bar{z}_{1}^{n-1}}\d z'\wedge\d\bar{z}'\\
&=-2\pi\i\int_{D}\Res_{A}(\omega\wedge\xi).\qedhere
\end{align*}
\end{proof}

%\begin{remark}
%\begin{description}
%\item{(a)} It would be nice to have a residue map $H^{p,q}_{A}(*\bar{*}D)\rightarrow H^{p-1,q-1}_{A}(*\bar{*}D)$. If we look at the form
%\[\omega=\frac{\d z_{1}\wedge\d\bar{z}_{1}}{z_{1}\bar{z}_{1}}+\frac{\d z_{1}\wedge\d\bar{z}_{2}}{z_{1}\bar{z}_{1}}\]
%in $\C^{2}$ we see that it is $\partial\bar{\partial}\omega$-closed; it is even $\partial$-closed. It would be tempting to let $\Res_{A}(\omega)=1$. However, changing coordinates by $w_{1}=z_{1}, w_{2}=z_{1}+z_{2}$ we get
%\[\omega=\frac{\d w_{1}\wedge\d\bar{w}_{1}}{w_{1}\bar{w}_{1}}+\frac{\d w_{1}\wedge\d(\bar{w}_{2}-\bar{w}_{1})}{w_{1}\bar{w}_{1}}\]
%from which it seems as the residue should be $0$. This causes problems if we want to define a residue for these forms.
%\item{(b)} We could imagine having a residue map $H^{d,d}_{A}(X\setminus D)\rightarrow H^{d-1,d-1}_{A}(D)$ but this also causes problems. If $d=1$ and
%\[\gamma=\log|z|^{2}\frac{\d\bar{z}}{\bar{z}}\]
%then
%\[\partial\gamma=\frac{\d z\wedge\d\bar{z}}{z\bar{z}}.\]
%But we should have $\Res_{A}(\partial\gamma)=0$.
%\end{description}
%\end{remark}

We can define the Aeppli residue for $(d,d)$-forms which have poles along a hypersurface with normal crossings as follows. Suppose $D=D_{1}\cup\dots\cup D_{k}$ for smooth hypersurfaces $D_{1},\dots,D_{k}$ and that $\omega\in H^{d,d}_{A}(*\bar{*}D)$. Considering $\omega$ on $X\setminus D$ we may define its residue with respect to the hypersurface $D_{1}\setminus\big(D_{2}\cup\dots\cup D_{k}\big)$ and we denote it $\Res_{A}^{D_{1}}(\omega)$. We should note here that, even though $X\setminus D$ is not compact, we can define the residue since the orders of the poles of $\omega$ are bounded, cf.\! the remark after Lemma \ref{lemmaConj}.

The residue $\Res_{A}^{D_{1}}(\omega)$ is represented by a form which has poles along the hypersurfaces $D_{1}\cap D_{i}$ and so in particular $\Res_{A}^{D_{1}}(\omega)\in H^{d-1,d-1}_{A}(*\bar{*}D_{\mathrm{sing}})$. We can make the same construction for every $D_{i}$ and then let
\[\Res_{A}^{D}(\omega)=\Res_{A}^{D_{1}}(\omega)+\dots+\Res_{A}^{D_{k}}(\omega).\]
By iterating this construction for the hypersurfaces $D_{i}\cap D_{j}$ in $D$ and so on we may define the Aeppli residues for all normal crossings. In particular, writing $E=D_{1}\cap\dots\cap D_{k}$, we get a residue $\Res_{A}^{E}(\omega)$ which is now represented by a smooth form. We also set $\Res_{A}^{X}(\omega)=\omega$.

We get the following generalisation of Theorem \ref{thmAeppli}.

\begin{theorem}\label{thmAeppli2}
For $\omega\in\E(*\bar{*}D)$ such that $D$ has normal crossings we have
\[\big\langle\{\omega\},\xi\big\rangle_{X}=(-2\pi\i)^{\kappa(\omega)}\Big\langle\big\{\Res_{A}^{E(\omega)}(\omega\wedge\xi)\big\},1\Big\rangle_{E(\omega)}.\]
\end{theorem}

\begin{remark}
In the above theorem we take the canonical current of a cohomology class which is \emph{not} a well defined object. However, its action on $1$ is.
\end{remark}

\begin{proof}
Take a partition of unity with the same properties as in the proof of Theorem \ref{thmAeppli}, but now the hypersurface will be given by $z^{I}=0$. Suppose $E(\omega)$ is given by $z_{1}=\dots=z_{\ell}=0$. Then we let $\d z'=\d z_{\ell+1}\wedge\dots\wedge\d z_{d}$. Let $R$ the multi-index which is $1$ in the $\ell$ first positions and otherwise $0$. If we write $p=2\kappa(\omega)+p'$ then
\[p'=\#\{j:J_{j}=0,K_{j}\neq0\}+\#\{K_{j}\neq0,J_{j}=0\}.\]
Now, similar to the proof of Theorem \ref{thmAeppli}, we get
\begin{align*}
&\quad\big\langle\{\omega\},\xi\big\rangle\\
&=\sum_{\iota}\frac{(-1)^{p}}{(J-1_{J})!(K-1_{K})!}\int_{\Delta}\Big(\prod_{j:J_{j}+K_{j}\neq0}\log|z_{j}|^{2}\Big)\frac{\partial^{J+K}\rho_{\iota}\psi}{\partial z^{J}\partial\bar{z}^{K}}\d z\wedge\d\bar{z}\\
&=(-2\pi\i)^{\kappa(\omega)}\sum_{\iota}\frac{(-1)^{2\kappa(\omega)+p'}}{(J-1_{J})!(K-1_{K})!}\int_{\Delta\cap E(\omega)}\!\!\!\Big(\!\!\!\prod_{\substack{j:J_{j}=0,K_{j}\neq0\\\text{or }J_{j}\neq0,K_{j}=0}}\!\!\!\!\!\!\!\!\!\!\log|z_{j}|^{2}\Big)\frac{\partial^{J+K-2R}\rho_{\iota}\psi}{\partial z^{J-R}\partial\bar{z}^{K-R}}\d z'\wedge\d\bar{z}'\\
&=(-2\pi\i)^{\kappa(\omega)}\sum_{\iota}(-1)^{p'}\int_{\Delta\cap E(\omega)}\!\!\!\Big(\!\!\!\prod_{\substack{j:J_{j}=0,K_{j}\neq0\\\text{or }J_{j}\neq0,K_{j}=0}}\!\!\!\!\!\!\!\!\!\!\log|z_{j}|^{2}\Big)\Res_{A}^{E(\omega)}(\omega\wedge\xi\rho_{\iota})\d z'\wedge\d\bar{z}'\\
&=(-2\pi\i)^{\kappa(\omega)}\big\langle\{\Res_{A}^{E(\omega)}(\omega\wedge\xi)\},1\big\rangle_{E(\omega)}.\qedhere
\end{align*}
\end{proof}

The right hand side of Theorem \ref{thmAeppli2} is a bit messy but with one extra assumption we get a cleaner statement.

\begin{corollary}\label{corMain}
For $\omega\in\E(*\bar{*}D)$ such that $D$ has normal crossings and $P^{1,0}(\omega)=P^{0,1}(\omega)$ we have
\[\big\langle\{\omega\},\xi\big\rangle=(-2\pi\i)^{\kappa(\omega)}\int_{E(\omega)}\Res^{E(\omega)}_{A}(\omega\wedge\xi).\]
\end{corollary}
\begin{proof}
Under these assumptions $\Res_{A}^{E(\omega)}(\omega\wedge\xi)$ is smooth on $E(\omega)$ so the statement follows from Theorem \ref{thmAeppli2}.
\end{proof}

%%%%%%%%%%%%%%%%%%%%%%
\section{Analytic continuation of divergent integrals}
%%%%%%%%%%%%%%%%%%%%%%

We will use the results in the previous sections to describe asymptotic expansions coming from analytic continuations of divergent integrals. In this section we drop the point of view of currents of quasi-meromorphic forms. Instead we suppose we have two semi-meromorphic forms $\alpha$ and $\beta$, on a compact complex manifold $X$, which have poles along the same hypersurface $D$. As before we assume $D$ to have normal crossings. We write
\[D_{d}\subset\dots\subset D_{1}\subset D_{0}\]
for the natural stratification of $D$, cf.\! (\ref{eqStratification}) in Section 2. Recall that $D_{0}=X$ and $D_{1}=D$. Regularising the integral
\[\int_{X}\alpha\wedge\bar{\beta}\]
we use Theorem \ref{thmPrinc} to get the asymptotic expansion
\[\int_{X}|s|^{2\lambda}\alpha\wedge\bar{\beta}=\lambda^{-\kappa}C_{-\kappa}+\dots+\lambda^{-1}C_{-1}+C_{0}+\O\big(|\lambda|\big)\]
where $\kappa=\kappa(\alpha\wedge\bar{\beta})$. Interpreting Corollary \ref{corMain} in this setting we get
\[C_{-\kappa}=\frac{(-2\pi\i)^{\kappa}}{o(s)}\int_{D_{\kappa}}\Res_{A}\big(\alpha\wedge\bar{\beta}\big)\]
where $o(s)=o_{\alpha\wedge\bar{\beta}}(s)$. We will now make some calculations of the other coefficients and we will in particular see how they depend on the metric. The coefficients also depend on the choice of section but as long as we do not change the line bundle this can be seen as a change of metric. The result is the following theorem.

\begin{theorem}\label{thmAsymp}
For the coefficients $C_{-r}$ in the asymptotic expansion
\[\int_{X}|s|^{2\lambda}\alpha\wedge\bar{\beta}=C_{-\kappa}\lambda^{-\kappa}+\dots+C_{-1}\lambda^{-1}+C_{0}+\O\big(|\lambda|\big)\]
we have
\begin{description}
\item{(a)} $C_{-r}$ depends polynomially of degree $\kappa-r$ on the metric. More precisely, if $\phi$ is the difference of two metrics then there are differential operators $Q_{r,j}$ with integrable coefficients such that
\[C_{-r}(\phi)=\sum_{j=0}^{\kappa-r}\int_{X}Q_{r,j}(\phi^{j}).\]
\item{(b)} The term $\int_{X}Q_{r,\kappa-r}(\phi^{\kappa-r})$ may be written
\[\frac{(-2\pi\i)^{\kappa}(-2)^{\kappa-r}}{o(s)(\kappa-r)!}\int_{D_{\kappa}}\Res_{A}\big(\phi^{\kappa-r}\alpha\wedge\bar{\beta}\big),\]
\item{(c)} $C_{-r}$ may be written as an integral over $D_{r}$, i.e.\! the codimension $r$ components in the stratification of $D$.
\end{description}
\end{theorem}

\begin{proof}
Similarly as in Section 2.2 we let
\begin{align*}
F(\lambda)&=o(s)\int_{X}|s|^{2\lambda}\alpha\wedge\bar{\beta}
\end{align*}
and from the proof of Theorem \ref{thmPrinc} we get
\[F(\lambda)=\frac{(-1)^{|J|+|K|}o(s)}{\lambda^{p}}h(\lambda)g(\lambda)\]
where
\[g(\lambda)=\sum_{\iota}\int_{\Delta}|z^{I}|^{2\lambda}\frac{\partial^{J+K}}{\partial z^{J}\partial\bar{z}^{K}}\big(\e^{-2\lambda\phi}\psi_{\iota}\big)\d z\wedge\d\bar{z},\]
$\psi_{\iota}$ is given by $\big(\psi_{\iota}/(z^{J}\bar{z}^{K})\big)\d z\wedge\d\bar{z}=\rho_{\iota}\alpha\wedge\bar{\beta}$ and $h$ and $p$ is given by Lemma \ref{lemmaMulti}. We may choose $J$ and $K$ independent of $\iota$. From now on we will suppress $\iota$ and $\rho_{\iota}$. Since we have assumed that $\alpha$ and $\beta$ have poles along the same hypersurface $p=2\kappa$. From the proof of Theorem \ref{thmPrinc} we know that $g^{(k)}(0)=0$ for $k=0,\dots,p-\kappa-1$. Taylor expanding $hg$ we get, for $r=0,1\dots,\kappa$,
\[C_{-r}=\frac{(-1)^{|J|+|K|}}{(p-r)!}\sum_{k=p-\kappa}^{p-r}{{p-r}\choose{k}}h^{(p-r-k)}(0)g^{(k)}(0).\]
Lemma \ref{lemmaMulti} implies that the derivatives of $h$ are combinatorial expressions involving $J$ and $K$. From the proof of Theorem \ref{thmPrinc} we also get
\begin{align*}
g^{(k)}(0)=\sum_{\ell=\kappa}^{k}{{k}\choose{\ell}}(-2)^{k-\ell}\sum_{M}{{\ell}\choose{M}}\int_{\Delta}\prod_{j=1}^{d}\big(I_{j}\log|z_{j}|^{2}\big)^{M_{j}}\frac{\partial^{J+K}}{\partial z^{J}\partial\bar{z}^{K}}\big(\psi\phi^{k-\ell}\big)\d z\wedge\d\bar{z}
\end{align*}
and hence we have proven the first part of (a), that $C_{-r}=\int_{X}\sum Q_{r,j}(\phi^{j})$ for some differential operators $Q_{r,j}$. We further see that the highest power of $\phi$ is obtained when $k$ is as large as possible and $\ell$ is as small as possible. Thus setting $k=p-r$, $\ell=\kappa$ and collecting the constants we get that the leading term is given by
\begin{align*}
&\frac{(-1)^{|J|+|K|}(-2)^{\kappa-r}}{(\kappa-r)!}h(0)\int_{\Delta}\prod_{j=1}^{d}\big(I_{j}\log|z_{j}|^{2}\big)^{M_{j}}\frac{\partial^{J+K}}{\partial z^{J}\partial\bar{z}^{K}}\big(\psi\phi^{\kappa-r}\big)\d z\wedge\d\bar{z}\\
&=\frac{(-2\pi\i)^{\kappa}(-2)^{\kappa-r}}{o(s)(\kappa-r)!}\int_{D_{\kappa}}\Res_{A}(\phi^{\kappa-r}\alpha\wedge\bar{\beta})
\end{align*}
if we do a similar calculation as in the proof of Proposition \ref{PropLocal}. This proves the rest of (a) and (b).

To prove (c) we may suppose that $I_{1},\dots,I_{\kappa}\neq0$ and $I_{\kappa+1},\dots,I_{d}=0$. We must show that we can reduce all the integrals in all the derivatives of $g$ to an integral over $D_{r}$. Let us look at $g^{(k)}$ for $k=\kappa,\dots,p-r$. In the expression for the derivative we have a multi-index $M$ such that $\sum_{j}M_{j}=\ell$, where $\ell\leq k$. We have seen that when $M_{i}=1$, so that we have $\log|z_{i}|^{2}$ in the integral, we may reduce it to an integral over $\Delta\cap\{z_{i}=0\}$.

First let $M_{1}=\dots=M_{\kappa}=1$. But then we need to add $\ell-\kappa$ to these indices, i.e.\! at most we need to add $p-r-\kappa=\kappa-r$. But if we add $1$ to $\kappa-r$ different $M_{j}$ there are still $r$ number of $M_{j}$ which are equal to one. Furthermore, in these variables we may reduce the integrals $r$ times, hence to codimension $r$. Adding more than one to some $M_{j}$ only makes it better.
\end{proof}

Theorem \ref{thmAsymp} points out why we call the currents defined from quasi-meromorphic forms  canonical; the currents come from the only coefficient in the asymptotic expansion which is independent of the metric. In the special case that $D$ is a smooth hypersurface we get the following corollary.

\begin{corollary}
If $D$ is a smooth hypersurface then 
\[\int_{X}|s|^{2\lambda}\alpha\wedge\bar{\beta}=\lambda^{-1}C_{-1}+C_{0}+\O\big(|\lambda|\big)\]
with $C_{-1}=-\frac{2\pi\i}{o(s)}\int_{D}\Res_{A}(\alpha\wedge\bar{\beta})$ and
\[C_{0}(\phi)=\frac{4\pi\i}{o(s)}\int_{D}\Res_{A}\big(\phi\alpha\wedge\bar{\beta}\big).\]
\end{corollary}

%This result can be used to retrieve the result of Felder and Kazhdan via the Mellin transform:
%
%\begin{align*}
%\mathcal{M}\big(G\big)(\lambda)&=\int_{0}^{\infty}x^{\lambda-1}G(x)\d x\\
%&=\int_{X}\alpha\wedge\bar{\beta}\int_{0}^{h^{2}}x^{\lambda-1}\d x\\
%&=\frac{1}{\lambda}\int_{X}h^{2\lambda}\alpha\wedge\bar{\beta}\\
%&=\frac{1}{\lambda}F(\lambda).
%\end{align*}

%%%%%%%%%%%%%%%%%%%%%%
%\section*{References}
%%%%%%%%%%%%%%%%%%%%%%

%\printbibliography[heading=none]
%\nocite{*}
%\printbibliography
\bibliography{bibliography}

\begin{thebibliography}{12}
\providecommand{\natexlab}[1]{#1}
\providecommand{\url}[1]{\texttt{#1}}
\expandafter\ifx\csname urlstyle\endcsname\relax
  \providecommand{\doi}[1]{doi: #1}\else
  \providecommand{\doi}{doi: \begingroup \urlstyle{rm}\Url}\fi

\bibitem[Andersson(2004)]{And}
M.~Andersson.
\newblock Residue currents and ideals of holomorphic functions.
\newblock \emph{Bull. Sci. math.}, \penalty0 (128), 2004.

\bibitem[Angella and Tomassini(2013)]{Ang2}
D.~Angella and A.~Tomassini.
\newblock On the $\partial\dbar$-lemma and bott--chern cohomology.
\newblock \emph{Invent. Math.}, 192\penalty0 (1):\penalty0 71--81, 2013.

\bibitem[Angella and Tomassini(2015)]{Ang1}
D.~Angella and A.~Tomassini.
\newblock Inequalities {\`a} la fr{\"o}licher and cohomological decompositions.
\newblock \emph{J. Noncommut. Geom.}, 9\penalty0 (2), 2015.

\bibitem[Barlet(1982)]{Bar}
D.~Barlet.
\newblock D{\'e}veloppement asymptotique des fonctions obtenues par
  int{\'e}gration sur les fibres.
\newblock \emph{Invent. math.}, \penalty0 (68), 1982.

\bibitem[Barlet and Maire(1989)]{Bar2}
D.~Barlet and H.-M. Maire.
\newblock Asymptotic expansion of complex integrals via mellin transform.
\newblock \emph{Journal of Functional Analysis}, \penalty0 (83), 1989.

\bibitem[Berenstein et~al.(1989)Berenstein, Gay, and Yger]{Ber}
C.~Berenstein, R.~Gay, and A.~Yger.
\newblock Analytic continuation of currents and division problems.
\newblock \emph{Forum Math.}, 1:\penalty0 15--51, 1989.

\bibitem[Chirka et~al.(1985)Chirka, Dolbeault, Khenkin, and Vitushkin]{Ch}
E.~M. Chirka, P.~Dolbeault, G.~Khenkin, and A.~Vitushkin.
\newblock \emph{Introduction to Complex Analysis}.
\newblock Springer, 1985.

\bibitem[Deligne et~al.(1975)Deligne, Griffiths, Morgan, and Sullivan]{Del}
P.~Deligne, P.~Griffiths, J.~Morgan, and D.~Sullivan.
\newblock Real homotopy theory of k{\"a}hler manifolds.
\newblock \emph{Invent. Math.}, 29\penalty0 (3):\penalty0 245--274, 1975.

\bibitem[Demailly()]{Dem}
J.-P. Demailly.
\newblock \emph{Complex Analytic and Differential Geometry}.
\newblock Available at
  \\http://www-fourier.ujf-grenoble.fr/$\sim$demailly/\\manuscripts/agbook.pdf.

\bibitem[Felder and Kazhdan(2016)]{Fe}
G.~Felder and D.~Kazhdan.
\newblock Divergent integrals, residues of dolbeault forms and asymptotic
  riemann mappings.
\newblock \emph{International mathematics research notices}, \penalty0 (08),
  2016.

\bibitem[Felder and Kazhdan(2018)]{Fe2}
G.~Felder and D.~Kazhdan.
\newblock Regularization of divergent integrals.
\newblock \emph{Selecta Mathematica}, 24\penalty0 (1), 2018.

\bibitem[Samuelsson(2009)]{Sa}
H.~Samuelsson.
\newblock Analytic continuation of residue currents.
\newblock \emph{Ark. Mat.}, \penalty0 (47), 2009.

\end{thebibliography}

\end{document}